\documentclass[a4paper,12pt]{amsart}

\usepackage{amsfonts}
\usepackage{amsmath}
\usepackage{amssymb}
\usepackage{mathrsfs}
\usepackage[colorlinks]{hyperref}

\usepackage{graphicx}

\usepackage[usenames]{color}

\newtheorem{thm}{Theorem}[section]

\newtheorem{prm}[thm]{Problem}
\numberwithin{equation}{section}

\theoremstyle{definition}
\newtheorem{definition}[thm]{Definition}
\newtheorem{rem}[thm]{Remark}

\begin{document}

\title[Well-posedness of differential-difference equations]
{Well-posedness of heat and wave equations generated by Rubin's $q$-difference operator in Sobolev spaces}

\author[Serikbol Shaimardan]{Serikbol Shaimardan}
\address{
  Serikbol Shaimardan:
  \endgraf
  L. N. Gumilyov Eurasian National University
  \endgraf
  5 Munaytpasov str., Astana, 010008
  \endgraf
  Kazakhstan
  \endgraf
  {\it E-mail address} {\rm shaimardan.serik@gmail.com}
  }

\author[Lars-Erik Persson]{Lars-Erik Persson}
\address{
  Lars-Erik Persson:
  \endgraf
  Department of Computer Science and Computational Engineering
  \endgraf
  UiT The Arctic University of Norway, Campus Narvik
  \endgraf
  Narvik, Norway
  \endgraf
   and
  \endgraf
  Department of computer science and mathematics Karlstad university
  \endgraf
  Karlstad, Sweden.
  \endgraf
  {\it E-mail address} {\rm larserik6pers@gmail.com}
  }

\author[N. Tokmagambetov ]{Niyaz Tokmagambetov }
\address{
  Niyaz Tokmagambetov:
  \endgraf 
  Centre de Recerca Matem\'atica
  \endgraf
  Edifici C, Campus Bellaterra, 08193 Bellaterra (Barcelona), Spain
  \endgraf
  and
  \endgraf   
  Institute of Mathematics and Mathematical Modeling
  \endgraf
  125 Pushkin str., 050010 Almaty, Kazakhstan
  \endgraf  
  {\it E-mail address:} {\rm tokmagambetov@crm.cat; tokmagambetov@math.kz}
  }

\thanks{SS  was supported in parts by the MESRK (Ministry of Education and Science of the Republic of Kazakhstan) grant AP08052208.}

\date{}

\begin{abstract}
In this paper, we investigate difference-differential operators of parabolic and hyperbolic types. Namely, we consider non-homogenous heat and wave equations for Rubin's difference operator. Well-posedness results are obtained in appropriate Sobolev type spaces. In particular, we prove that the heat and wave equations generated by Rubin's difference operator have unique solutions. We even show that these solutions can be represented by explicit formulas.

\end{abstract}



\maketitle

\section{Introduction}

In the theory of partial differential equations the heat and wave equations are most important and basic examples. There are a lot of papers studying these equations in different modifications. Also, in this paper we aim to study a $q$-analogues of the following classical heat and wave equations
\begin{equation*}
\left\{
  \begin{array}{ll}
    u_t(t, x)-\Delta u(t, x)+mu(t, x)=f(t, x), & (t, x)\in\hbox{$\mathbb{R^+}\times\mathbb{R}^n$,} \\
    u(0, x)=\varphi(x),   &       \hbox{$x\in \mathbb{R}^n$,}
  \end{array}
\right.
\end{equation*}
and 
\begin{equation*}
\left\{
  \begin{array}{ll}
    u_{tt}(t, x)-\Delta u(t, x)+bu_{t}(t, x)+mu(t, x)=f(t, x), & (t, x)\in\hbox{$\mathbb{R^+}\times\mathbb{R}^n$,} \\
    u(0, x)=\varphi(x),   u_t(0, x)=\psi(x), & \hbox{$x\in \mathbb{R}^n$},
  \end{array}
\right.
\end{equation*}
 where $b, m\in \mathbb{R}$
  and $\Delta$ is the Laplace operator in
$\mathbb{R}^n$.

Here, we will consider the above equations with Rubin's type $q$-difference operators instead of $\Delta$. If we go back to the history, we see that already in 1748 L. Euler \cite{Euler} introduced basic elements of $q$-calculus. He considered the infinite product $(q;q)_\infty^{-1}=\prod\limits_{k=0}^\infty\frac{1}{1-q^{k+1}}$, $|q|<1$. Around hundred years later  the progress  continued under E. Heine, who in 1846 considered a generalization of the hypergeometric ($q$-hypergeometric) series (see \cite{Gasper} and \cite{Heine}). In the second part of the twentieth  century the $q$-calculus served as a bridge between physics and mathematics.  The $q$-calculus has numerous applications in various fields of mathematics e.g number theory, combinatorics, special functions,  harmonic analysis, fractional calculus and also for scientific problems in some applied areas such as computer science, quantum mechanics and quantum physics (see e.g. 
  \cite{Annaby}, \cite{AKAMST2021}, \cite{CWM2021}, \cite{Ernst1}, \cite{Exton},\cite{Fitouhi} and  \cite{Maligranda}). For the further development and recent results in $q$-calculus we refer to the books \cite{Agarwal} and  \cite{Ernst}  and the references given therein.

In the beginning of the last century, especially F.H. Jackson  \cite{Jackson1} (see also \cite{Jackson}), R.D. Carmichael \cite{Carmichael}, T.E. Mason \cite{Mason}, and C.R. Adams \cite{Adams1} (see also \cite{Adams2}) intensively studied the $q$-difference equations. Thereafter, the $q$-difference equations are used to modelling some important linear and nonlinear problems and thereby played an important role in different fields of engineering and science \cite{Bangerezako1}(see also \cite{Bangerezako2}). Moreover, there has been a great interest in finding difference methods to study exact or approximate solutions of ordinary and partial $q$-differential equations,  see e.g. \cite{Area}, \cite{Jian}, \cite{DaJian}, \cite{Jafari}, \cite{Jiang} and \cite{Zheng}.

In 1997, R.L. Rubin \cite{Rubin2} studied spectral properties of one $q$-difference operator (Rubin’s operator). By using established spectral properties, he constructed the $q$-difference calculus associated with the Rubin’s operator. A $q^2$-analogue of the exponential Fourier analysis via $q^2$-Fourier transform was constructed in \cite{Koornwinder}. There the author used analogues of trigonometric functions (or orthogonality results) and applied it to $q$-deformed quantum mechanics \cite{Schoichi03}. Moreover, he used the Dalembert and Duhamel's techniques to investigation the wave $q$-difference equation, using the $q^2$-Fourier multiplier tools to aid the construction of solutions in \cite{Rubin1}. The papers \cite{Saoudi1},  \cite{Saoudi2} and \cite{Fitouhi2012} were devoted to the study of the $q$-analogue of the Fourier transform and to show how it plays a central role in solving the $q$-heat and $q$-wave equations associated with the Rubin’s $q$-difference operator $\partial_q$.

However, the theory of the $q$-difference equations with the Rubin’s $q$-difference operator $\partial_q$ is still at the initial stage and many aspects of this theory need to be explored. To the best of our knowledge, the theory of the Cauchy problems for linear, homogeneous and nonhomogeneous $q$-difference equations generated by the Rubin’s $q$-difference operator are not yet developed.

From the other side, it is well known that the quantum calculus provides natural discrete modifications of the heat and wave equations. To discretize, we will replace the partial derivatives  $\frac{\partial{u}}{\partial{x}} $  by the Rubin’s $q$-difference operators $\partial_q$ in space, respectively, and we will attempt to develop the $q$-analogues of these problems. Here, as we take a limit as $q$ tends to $1$, one recovers results related to the continuous heat and wave equations.

Motivated by this, we study the Cauchy problems for the non-homogenous heat and wave equations generated by the Rubin's difference operator $\mathcal{D}^2_{q,x}$ in the Sobolev space $L^2_q\left(\mathbb{R}^+_q\right)$
\begin{equation*}
\left\{
  \begin{array}{ll}
    u_t(t, x)-\mathcal{D}^2_{q,x}u(t, x)+mu(t, x)=f(t, x), & (t, x)\in\hbox{$\mathbb{R^+}\times\mathbb{R}^+_q$,} \\
    u(0, x)=\varphi(x),   &       \hbox{$x\in \mathbb{R}^+_q$,}
  \end{array}
\right.
\end{equation*}
and 
\begin{equation*}
\left\{
  \begin{array}{ll}
    u_{tt}(t, x)-\mathcal{D}^2_{q,x}u(t, x)+bu_{t}(t, x)+mu(t, x)=f(t, x), & (t, x)\in\hbox{$\mathbb{R^+}\times\mathbb{R}^+_q$,} \\
    u(0, x)=\varphi(x),   u_t(0, x)=\psi(x), & \hbox{$x\in \mathbb{R}^+_q$},
  \end{array}
\right.
\end{equation*}
here the $q$-partial differential operator \cite{Saoudi2, Rubin1} is defined by
\begin{eqnarray*}
\mathcal{D}_{q,x}u(t,x):=\frac{u(t,q^{-1}x)-u(t,qx)+u(t,-qx)-u(t,-x)+u(t,-q^{-1}x)-u(t,-x)}{2x(1-q)}.
\end{eqnarray*}

The method developed in this paper is different from that given in \cite{Saoudi2} and \cite{Rubin2}, in our opinion, simpler and more clear and, in addition give the possibility to prove uniqueness and derive explicit formulas for the solution.

\subsection{Structure of the  paper}
In Section \ref{S3} we prove that these equations both has a unique solution and also derive an explicit formula for this solution in both cases (see Theorems \ref{TH1} and \ref{TH2}, respectively). In order to illustrate that our method can  work also for other equations in  Section \ref{S4} we
prove a corresponding uniqueness result for a non-homogeneous equation with homogeneous boundary conditions. Also in this case we derive an
explicit formula for the solution (see Theorem \ref{TH3}). In order not to disturb our discussions in these two main Sections we present some necessary
preliminaries   in Section \ref{Pre}.

\section{Preliminaries}
\label{Pre}

We start by recalling some basic notation in the $q$-calculus, see e.g.  the books \cite{Cheung} and  \cite{Ernst}. Throughout this paper, we assume that $0<q<1$.

Let $\alpha \in \mathbb R$. Then a $q$-real number $[\alpha ]_{q}$ is defined by
$$
[\alpha ]_{q}:=\frac{1-q^{\alpha }}{1-q},
$$
where $\mathop{\lim }\limits_{q\rightarrow 1}\frac{1-q^{\alpha }}{1-q}%
=\alpha $.

We introduce for any $x, a\in\mathbb R$
\begin{eqnarray*}
(x,a)^0_q=1, \;\;\; (x,a)^n_q=\prod\limits_{k=0}^n\left(x-q^ka\right), \;\;\;(x,a)^\infty_q&=&\lim\limits_{n\rightarrow\infty}(x,a)^n_q.
\end{eqnarray*}

The $q$-analogues of the binomial coefficients and the Gamma function are defined by
\begin{equation*}
[n]_{q}!:=\left\{
\begin{array}{l}
1,\,\,\,\,\,\,\,\,\,\,\,\,\,\,\,\,\,\,\,\,\,\,\,\,\,\,\,\,\,\,\,\,\,\,\,\,\,\,\,\,\,\,\,\,\,\,\,\,\,\,\,\,\, \hbox{if} \,\,\, n=0 \\
 {[1]_{q}\times [2]_{q}\times \cdots \times [n]_{q} \,\,\, \hbox{if} \,\,\, n\in \mathbb N}
\end{array}
\right.
\end{equation*}
and
\begin{eqnarray*}
\Gamma_q(x)=\frac{(q, q)^\infty_q}{(q^x, q)^\infty_q}(1-q)^{1-x},
\end{eqnarray*}
respectively.

The $q^2$-exponentials $e_{q^2}(z)$ (see \cite{Koornwinder}, \cite{Rubin2} and \cite{Rubin1}) are defined by
\begin{eqnarray*}
e_{q^2}(z):=\cos\limits_{q^2}\left(-iz\right)+i\sin\limits_{q^2}\left(-iz\right),
\end{eqnarray*}
where
\begin{eqnarray*}
\cos\limits_{q^2}(z):=\sum\limits_{k=0}^\infty\frac{(-1)^kq^{k(k+1)}z^{2k}}{[2k]_{q}!},\\
\end{eqnarray*}
and
\begin{eqnarray*}
\sin\limits_{q^2}(z):=\sum\limits_{k=0}^\infty\frac{(-1)^kq^{k(k+1)}z^{2k+1}}{[2k+1]_{q}!}
\end{eqnarray*}
are the $q^2$-trigonometric functions. Note that the series defining $e_{q^2}(z)$ is absolutely convergent for all $z$ in the plane, $0 <q <1$, since both of  the  series defining its component functions are absolutely convergent. Moreover, it yields that $\lim\limits_{q\rightarrow1} e_{q^2}(z) = e(z)$  pointwise and uniformly on compact  sets, because both of its component functions satisfy the corresponding limits
($\lim\limits_{q\rightarrow1} \cos\limits_{q^2}(z) = \cos (z)$ and $\lim\limits_{q\rightarrow1} \sin\limits_{q^2}(z) = \sin (z)$).

The $q$-analogue differential operator is defined by (see \cite{Jackson1})
\begin{eqnarray*}
D_{q}f(x)=\frac{f(x)-f(qx)}{x(1-q)},
\end{eqnarray*}
and the $q^2$-differential operator or Rubin operator is defined by (see \cite{Rubin2} and \cite{Rubin1})
\begin{eqnarray*}\label{additive2.5}
\mathcal{D}_{q}f(x)=\frac{f(q^{-1}x)+f(-q^{-1}x)
-f(qx)+f(-qx)-2f(-x)}{2x(1-q)}.
\end{eqnarray*}

Moreover, we will use the $q$-partial differential operators $\mathcal{D}_{q,x}u(t,x)$ in the following form:
\begin{eqnarray*}
\mathcal{D}_{q,x}u(t,x):=\frac{u(t,q^{-1}x)+u(t,-q^{-1}x)
-u(t,qx)+u(t,-qx)-2u(t,-x)}{2x(1-q)}.
\end{eqnarray*}

Note that if $f$ is differentiable at $x$, then $\lim\limits_{q\rightarrow1}\mathcal{D}_{q}f(x)=f'(x)$ and 
$\lim\limits_{q\rightarrow1}\mathcal{D}_{q,x}u(t,x)=\frac{\partial u}{\partial x}(t,x)$.

The definite $q$-integral or the $q$-Jackson integral of a function $f$ is defined by the formula (see \cite{Jackson1} and \cite{Jackson})
\begin{eqnarray}\label{additive2.1A}
\int\limits_0^xf(t)d_qt:=(1-q)x\sum\limits_{k=0}^\infty{q^k}f(q^kx),\quad 0<x<\infty,
\end{eqnarray}
and the improper $q$-integral of a function $f(x):[0, \infty)\rightarrow \mathbb{R}$, is defined by the formula
\begin{eqnarray}\label{additive2.1B}
\int\limits_0^{\infty} f(t)d_qt:=(1-q)\sum\limits_{k=-\infty}^\infty{q^k}f(q^k).
\end{eqnarray}
Note that the series in the right hand sides of (\ref{additive2.1A}) and (\ref{additive2.1B}) converge absolutely.

We denote $\mathbb{R}^+_q=\{q^k, k\in \mathbb{Z}\}$ and define
\begin{eqnarray*}
L^p_q\left(\mathbb{R}^+_q\right)&:=&\left\{f:\left(\int\limits_0^\infty |f(x)|^pd_qx\right)^\frac{1}{p}<\infty\right\},
\end{eqnarray*}
for $0<p<\infty$ and 
\begin{eqnarray*}
L^\infty_q\left(\mathbb{R}^+_q\right)&:=&\left\{f:\sup\limits_{k\in\mathbb{Z}}|f(q^k)|<\infty\right\}.
\end{eqnarray*}

Moreover,  we define $S_q\left(\mathbb{R}^+_q\right)$ as the set of functions $f$ defined on $\mathbb{R}^+_q$ satisfying the following condition:
$$
P_{n,m,q}\left(f\right):=\sup\limits_{x\in \mathbb{R}^+_q}\left|x^m\mathcal{D}^n_qf(x)\right|<\infty,
$$
for all $m, n \in \mathbb{N}$ and where $\lim\limits_{x\rightarrow0}\mathcal{D}^n_qf(x)$ exists in  $\mathbb{R}^+_q$ (see \cite{Fitouhi}).

We also define the space $S'_q\left(\mathbb{R}^+_q\right)$ as the topological dual of $S_q\left(\mathbb{R}^+_q\right)$. The space $S'_q\left(\mathbb{R}^+_q\right)$ is called the space of tempered distributions on $\mathbb{R}^+_q$.

\begin{definition} \label{Def:01}
(see \cite{Fitouhi} and \cite{Rubin1})
The $q^2$-Fourier transform $\hat{f}$ is  defined as follows
\begin{eqnarray*}
\hat{f}(\xi)&=&\frac{1}{2\pi_q}
\int\limits_0^\infty f(x)e_{q^2}(-i x \xi)d_{q} x, \;\;\; \xi\in\mathbb{R}^{+}_{q},
\end{eqnarray*}
for $f\in S_q\left(\mathbb{R}^+_q\right)$ and  its inverse $\check{g}(x)$ is given by
\begin{eqnarray*}
\check{g}(x)&=& \frac{1}{2\pi_q}
\int\limits_0^\infty e_{q^2}(i x \xi)g(\xi)d_q \xi, \;\;\; x\in\mathbb{R}^{+}_{q},
\end{eqnarray*}
for $g\in S_q\left(\mathbb{R}^+_q\right)$, where $\frac{1}{\pi_q}:=\frac{\left(1+q\right)^\frac{1}{2}}{\Gamma_{q^2}(\frac{1}{2})}$.
\end{definition}

 Note that the $q^2$-Fourier transform is an isomorphism from $L^2_q(\mathbb{R}^+_q)$ onto itself. In fact, we have the following
version of the Parseval identity (see e.g. \cite{Rubin2} and \cite{Rubin1})
\begin{eqnarray}\label{additive2.2}
\|\widehat{f}\|_{L^2_q(\mathbb{R}^+_q)}=\|f\|_{L^2_q(\mathbb{R}^+_q)},   \forall  f\in L^2_q(\mathbb{R}^+_q).
\end{eqnarray}
Moreover,  for $u\in S'_q\left(\mathbb{R}^+_q\right)$ it yields that (see e.g. \cite{Rubin1}):
\begin{eqnarray}\label{additive2.3}
\widehat{\partial^nu}=(i\xi)^n\widehat{u}, 
\end{eqnarray}
for all $\xi\in\mathbb{R}^{+}_{q},$ where $n\in\mathbb{N}$.

\begin{definition} (see \cite{Saoudi1}) For $s \in \mathbb{R}$, we define the Sobolev space $W_q^s\left(R^+_q\right)$ as
$$
W_q^s\left(R^+_q\right):=\left\{u\in S'_q\left(\mathbb{R}^+_q\right):
\left(1+\left|\xi\right|^2\right)^\frac{s}{2}\widehat{u}\in L^2_q\left(\mathbb{R}^+_q\right)\right\},
$$
equipped with the norms
$$
\|u\|^2_{W_q^s\left(R^+_q\right)}:=
\left(\int\limits_0^\infty\left(1+\left|\xi\right|^2\right)
^s|\widehat{u}(\xi)|^2d_q \xi\right)^2.
$$
\end{definition}

Let $0<T<\infty$. We introduce also the spaces  $C^k\left( [0, T]; W^s_q\left(\mathbb{R}^+_q\right)\right)$ and
$C^k\left( [0, T]; L^2_q\left(\mathbb{R}^+_q\right)\right)$  defined by the finiteness of the norms
$$
\|u\|_{C^k\left( [0, T]; W^s_q\left(\mathbb{R}^+_q\right)\right)}:=
\sum\limits_{n=0}^k\max\limits_{0\leq t\leq T}\|\partial_t^nu(t,.)\|_{W_q^s\left(R^+_q\right)}.
$$
and
$$
\|u\|_{C^k\left( [0, T]; L^s_q\left(\mathbb{R}^+_q\right)\right)}:=
\sum\limits_{n=0}^k\max\limits_{0\leq t\leq T}\|\partial_t^nu(t,.)\|_{L^2_q\left(\mathbb{R}^+_q\right)},
$$
respectively.

\textbf{Notation.} The symbol $M \lesssim K$ means that there exists $\gamma> 0$ such that
$M \leq \gamma K$, where $\gamma$ is a constant.

\section{Main results concerning the $q$-heat and $q$-wave  equations}
\label{S3}

Our first main result concerns the solvability of the following. 
\begin{prm}[$q$-Heat Equation]
\label{Pr1}
Let $m\in\mathbb{R}$ and $T>0$. We consider the following Cauchy problem (CP):
\begin{equation}
\label{additive3.1}
\left\{
 \begin{array}{ll}
u_t(t, x)-\mathcal{D}^2_{q,x}u(t, x)+mu(t, x)=f(t, x), & (t, x)\in\hbox{$[0, T]\times\mathbb{R}^+_q$,} \\
u(0, x)=\varphi(x),   
& \hbox{$x\in \mathbb{R}^+_q$.}
  \end{array}
\right.
\end{equation}
\end{prm}

\begin{thm} 
\label{TH1}
Assume that $m$ is positive and $T<\infty$. Let $\varphi \in W_q^2\left(\mathbb{R}^+_q\right)$ and $f\in C\left( [0, T]; W^2_q\left(\mathbb{R}^+_q\right)\right)$. Then there exists a unique solution of Problem \ref{Pr1}
$$
u\in C^1\left([0, T]; L^2_q\left(\mathbb{R}^+_q\right) \right) \cap C\left( [0, T]; W_q^2\left(\mathbb{R}^+_q\right)\right).
$$
Moreover, this solution can be represented by the  explicit formula
\begin{equation}
\label{additive3.1A}
\begin{split}
u(t,x) & = \frac{1}{4\pi_q^{2}}
\int\limits_{0}^{t} \int\limits_0^\infty \int\limits_0^\infty e^{ - (t-\tau)(\xi^2+m)}  f(\tau, x) e_{q^2}(-i\xi x) e_{q^2}(i\xi x) d_q x d_q\xi d\tau \\ 
& \,\,\,\,\,\,\,\,\,\,\,\,\,\,\,\,\,\,\,\, 
+ \frac{1}{4\pi_q^{2}}
\int\limits_0^\infty \int\limits_0^\infty e^{-t(m+\xi^2)} \varphi(x) e_{q^2}(-i\xi x) e_{q^2}(i\xi x) d_q x d_q\xi,
\end{split}
\end{equation}
for $(t, x)\in[0, T]\times\mathbb{R}^+_q$.
\end{thm}

\begin{proof} {\it Existence.} Let us fix $t\in[0, T]$. Then, by using (\ref{additive2.3})  we have that
\begin{eqnarray}\label{additive3.2}
\widehat{\mathcal{D}^2_{q,x}u}(t, \xi) &=&\frac{1}{2\pi_q}\int\limits_0^\infty \mathcal{D}^2_{q,x}u(t, x)e_{q^2}(-i x \xi)d_qx=-\xi^2\widehat{u}(t, \xi),
\end{eqnarray}
for all $(t, \xi) \in[0, T]\times\mathbb R_{q}^{+}$.

Now, by multiplying both sides of (\ref{additive3.1}) by $\frac{1}{2\pi_q}e_{q^2}(-i x \xi)$ and $q$-integrating with respect to $x$ we find that
\begin{equation}\label{additive3.3}
\widehat{u_t}(t, \xi)+\left(m+\xi^2\right)\widehat{u}(t, \xi)=\widehat{f}(t, \xi),
\end{equation}
for all $\xi\in\mathbb R_{q}^{+}$, and all $t\in [0, T]$.

Applying the $q^2$-Fourier transform and using the definition of the ordinary derivative, we get that
\begin{eqnarray}\label{additive3.4}
\frac{\partial}{\partial{t}}\widehat{u}(t, \xi)
&=&\lim\limits_{h\rightarrow0}\frac{\widehat{u}(t+h, \xi)-\widehat{u}(t, \xi)}{h}\nonumber\\
&=&\lim\limits_{h\rightarrow0}\frac{1}{2\pi_q}\int\limits_0^\infty\frac{u(t+h, \xi)-u(x,t)}{h}e_{q^2}(-i x \xi)d_qx\nonumber\\
&=&\frac{1}{2\pi_q}\int\limits_0^\infty\lim\limits_{h\rightarrow0}\frac{u(t+h, \xi)-u(x,t)}{h}e_{q^2}(-i x \xi)d_qx\nonumber\\
&=&\frac{1}{2\pi_q}\int\limits_0^\infty \frac{\partial{u}}{\partial{t}}(t, x)e_{q^2}(-i x \xi)d_qx\nonumber\\
&=&\widehat{u}_t(t, \xi).
\end{eqnarray}

From (\ref{additive3.3}) and (\ref{additive3.4}) it follows that
\begin{equation}\label{additive3.5}
\left\{
  \begin{array}{ll}
    \widehat{u}_t(t, \xi)+\xi^2\widehat{u}(t, \xi)+m\widehat{u}(t, \xi)=\widehat{f}(t, \xi), & (t, \xi)\in[0, T]\times\mathbb{R}^+_q, \\
    \widehat{u}(0, \xi)=\widehat{\varphi}(\xi),   & \xi\in \mathbb{R}^{+}_{q}.
  \end{array}
\right.
\end{equation}

The solution of the non--homogeneous ordinary differential equation with the constant coefficient
\begin{equation*}
\label{additive3.6}
\widehat{u}_{t}(t, \xi)+\left(\xi^2+m\right)\widehat{u}(t, \xi)=\widehat{f}(t, \xi),
\end{equation*}
which is satisfying the Cauchy data $\widehat{u}(0, \xi)=\widehat{\varphi}(\xi),$ is the following function
\begin{equation}
\label{additive3.7}
\widehat{u}(t,\xi) = \int\limits_{0}^{t} e^{ - (t-\tau)(\xi^2+m)}\widehat{f}(\tau,\xi)d\tau + \widehat{\varphi}(\xi) e^{-t(m+\xi^2)},
\end{equation}
for all $\xi\in\mathbb R_{q}^{+}$.

Finally, we obtain that
\begin{equation}
\label{additive3.8}
\begin{split}
u(t,x) & = \frac{1}{2\pi_q}
\int\limits_0^\infty\left[\int\limits_{0}^{t} e^{-(t-\tau)(\xi^2+m)}\widehat{f}(\tau,\xi)d\tau + e^{-t(m+\xi^2)} \widehat{\varphi}(\xi) \right]e_{q^2}(i\xi x)d_q\xi \\
& = \frac{1}{4\pi_q^{2}}
\int\limits_{0}^{t} \int\limits_0^\infty \int\limits_0^\infty e^{ - (t-\tau)(\xi^2+m)}  f(\tau, x) e_{q^2}(-i\xi x) e_{q^2}(i\xi x) d_q x d_q\xi d\tau \\ 
& \,\,\,\,\,\,\,\,\,\,\,\,\,\,\,\,\,\,\,\, 
+ \frac{1}{4\pi_q^{2}}
\int\limits_0^\infty \int\limits_0^\infty e^{-t(m+\xi^2)} \varphi(x) e_{q^2}(-i\xi x) e_{q^2}(i\xi x) d_q x d_q\xi,
\end{split}
\end{equation}
for all $(t, x)\in[0, T]\times\mathbb R_{q}^{+}$.

Next we will prove that $u\in C^1\left([0, T]; L^2_q\left(\mathbb{R}^+_q\right) \right) \cap C\left( [0, T]; W_q^2\left(\mathbb{R}^+_q\right)\right)$. By using (\ref{additive3.7}) and the Cauchy–Schwarz inequality, we can deduce that
\begin{equation}
\label{additive3.9}
\begin{split}
\left|\widehat{u}(t,\xi)\right|^2 
&\leq \left| \int\limits_{0}^{t} e^{ - (t-\tau)(\xi^2+m)}\widehat{f}(\tau,\xi)d\tau \right|^2 + \left|\widehat{\varphi}(\xi) e^{-t(m+\xi^2)}\right|^2 \\   
&\leq t \int\limits_{0}^{t} \left| e^{ - (t-\tau)(\xi^2+m)}\widehat{f}(\tau,\xi)\right|^2 d\tau  + e^{- 2 t m} \left|\widehat{\varphi}(\xi) \right|^2 \\   
&\leq t \int\limits_{0}^{t} \left| \widehat{f}(\tau,\xi)\right|^2 d\tau  + e^{- 2 t m} \left|\widehat{\varphi}(\xi) \right|^2,
\end{split}
\end{equation}
for all $\xi\in\mathbb R_{q}^{+}$.

Since $\varphi \in W_q^2\left(\mathbb{R}^+_q\right), f\in C\left([0, T]; W^2_q\left(\mathbb{R}^+_q\right) \right)$ and by using the Parseval's identity, we arrive at
\begin{equation}
\label{PI-1}
\begin{split}
\|u(t,\cdot)\|^{2}_{L_q^2\left(\mathbb{R}^+_q\right)} & = \sum_{\xi\in\mathbb R^{+}_{q}} \left|\widehat{u}(t,\xi)\right|^2 \\
& \leq t \int\limits_{0}^{t} \sum_{\xi\in\mathbb R^{+}_{q}} \left| \widehat{f}(\tau,\xi)\right|^2 d\tau  + e^{- 2 t m} \sum_{\xi\in\mathbb R^{+}_{q}} \left|\widehat{\varphi}(\xi) \right|^2 \\
& = t \int\limits_{0}^{t} \|f(\tau,\cdot)\|^{2}_{L_q^2\left(\mathbb{R}^+_q\right)} d\tau  + e^{- 2 t m}  \|\varphi\|^{2}_{L_q^2\left(\mathbb{R}^+_q\right)} \\
& \leq T^{2} \|f\|^{2}_{C\left([0, T]; L^2_q\left(\mathbb{R}^+_q\right) \right)} +   \|\varphi\|^{2}_{L_q^2\left(\mathbb{R}^+_q\right)} < \infty,
\end{split}
\end{equation}
for $T<\infty$. Thus, we can conclude that $\|u\|_{C\left([0, T]; L^2_q\left(\mathbb{R}^+_q\right) \right)}<\infty$.

Repeating the procedure above, we have that
\begin{equation}
\label{PI-2}
\begin{split}
\|u(t,\cdot)\|^{2}_{W_q^2\left(\mathbb{R}^+_q\right)} 
& = \sum_{\xi\in\mathbb R^{+}_{q}} \xi^{2} \left|\widehat{u}(t,\xi)\right|^2 \\
& \leq t \int\limits_{0}^{t} \sum_{\xi\in\mathbb R^{+}_{q}} \xi^{2} \left| \widehat{f}(\tau,\xi)\right|^2 d\tau  + e^{- 2 t m} \sum_{\xi\in\mathbb R^{+}_{q}} \xi^{2} \left|\widehat{\varphi}(\xi) \right|^2 \\
& = t \int\limits_{0}^{t} \|f(\tau,\cdot)\|^{2}_{W_q^2\left(\mathbb{R}^+_q\right)} d\tau  + e^{- 2 t m}  \|\varphi\|^{2}_{W_q^2\left(\mathbb{R}^+_q\right)} \\
& \leq T^{2} \|f\|^{2}_{C\left([0, T]; W^2_q\left(\mathbb{R}^+_q\right) \right)} +   \|\varphi\|^{2}_{W_q^2\left(\mathbb{R}^+_q\right)} < \infty,
\end{split}
\end{equation}
for $T<\infty$, and we find that $\|u\|_{C\left([0, T]; W^2_q\left(\mathbb{R}^+_q\right) \right)}<\infty$. 

In an analogical way, we can derive the estimate $\|u\|_{C^{1}\left([0, T]; L^2_q\left(\mathbb{R}^+_q\right) \right)}<\infty$, which ends proof of the    existence part. 

{\it Uniqueness.}  Let us obtain the requesteduniqueness by contradiction. Let us assume that there are two different solutions  $u(t,x)$ and $v(t,x)$ of Problem \ref{Pr1} such that
\begin{eqnarray*}
\left\{
  \begin{array}{ll}
    u_t(t,x)-\mathcal{D}^2_{q,x} u(t,x)+m u(t,x)=f(t,x), & (t,x)\in\hbox{$[0, T]\times\mathbb{R}^+_q$,} \\
    u(0,x)=\varphi(x),   & \hbox{$x\in \mathbb{R}^+_q$,}
  \end{array}
\right.
\end{eqnarray*}
and
\begin{eqnarray*}
\left\{
  \begin{array}{ll}
    v_t(t,x)-\mathcal{D}^2_{q,x}v(t,x)+mv(t,x)=f(t,x), & (t,x)\in\hbox{$[0, T]\times\mathbb{R}^+_q$,} \\
    v(0,x)=\varphi(x),    & \hbox{$x\in \mathbb{R}^+_q$.}
  \end{array}
\right.
\end{eqnarray*}

Denote $W(t, x)\equiv u(t, x)-v(t, x) $. Then the function $W(t, x)$ is the solution of the following problem.
\begin{eqnarray*}
\left\{
  \begin{array}{ll}
    W_{tt}(t,x)-\mathcal{D}^2_{q,x}W(t,x)+mW(t,x)=0, & (t, x)\in\hbox{$[0, T]\times\mathbb{R}^+_q$,} \\
    W(0, x)=0,    & \hbox{$x\in \mathbb{R}^+_q$.}
  \end{array}
\right.
\end{eqnarray*}

From  (\ref{additive3.8}) it follows that $W(t, x)\equiv0$. Hence, $u(t, x)\equiv v(t, x)$. This contradiction shows that our assumption is wrong so the solution is unique.  The proof is complete.

\end{proof}

Our next main result is to derive a corresponding result for the one dimensional $q$-Wave Equation.

\begin{prm}[$q$-Wave Equation]
\label{Pr2}
Let $b,m \in \mathbb{R^+}$. We consider the following homogeneous wave equation with the initial conditions
\begin{equation}
\label{additive3.1B}
\left\{
  \begin{array}{ll}
    u_{tt}(t,x)-\mathcal{D}^2_{q,x}u(t,x)+bu_{t}(t,x)+mu(t,x)=0, & (t,x)\in\hbox{$[0, T]\times\mathbb{R}^+_q$,} \\
    u(0,x)=\varphi(x),   u_t(0,x)=\psi(x), & \hbox{$x\in \mathbb{R}^+_q$},
  \end{array}
\right.
\end{equation}
where $0<T<\infty$.
\end{prm}

\begin{thm}
\label{TH2}
Fix $0<T<\infty$. Let $b>0$ and $m>0$ be such that $b^2<4m$. Assume that $\varphi \in W_q^2\left(\mathbb{R}^+_q\right)$ and $\psi\in W_q^1\left(\mathbb{R}^+_q\right)$. Then there exists a unique solution
$$
u\in C^2\left([0, T]; L^2_q\left(\mathbb{R}^+_q\right) \right)\cap C\left([0, T]; W_q^2\left(\mathbb{R}^+_q\right) \right),
$$
of Problem \ref{Pr2}. Moreover, this solution can be represented by the explicit formula
\begin{equation}
\label{additive3.2B}
u(t,x)=
\int\limits_0^\infty \Phi(t, x, y) \varphi(y) d_q y +\int\limits_0^\infty \Psi(t, x, y) \psi(y) d_q y, 
\end{equation}
where
\begin{equation*}
\begin{split}
\Phi(t, x, y) = \frac{e^{-\frac{b}{2} t}}{8\pi_q^{2}} \int\limits_0^\infty K(t, \xi) e_{q^2}( - i\xi y) e_{q^2}( i \xi x) d_q\xi,
\end{split}
\end{equation*}
with
$$ 
K(t, \xi) = e^{\frac{\sqrt{b^2 -4\left(m+\xi^2\right)}}{2} t} + e^{ - \frac{\sqrt{b^2 -4\left(m+\xi^2\right)}}{2} t} + b \frac{e^{\frac{\sqrt{b^2 -4\left(m+\xi^2\right)}}{2} t} - e^{ - \frac{\sqrt{b^2 -4\left(m+\xi^2\right)}}{2} t}}{\sqrt{b^2 -4\left(m+\xi^2\right)}},
$$
and
$$
\Psi(t, x, y) = \frac{e^{-\frac{b}{2} t}}{4\pi_q^{2}}\int\limits_0^\infty \frac{e^{\frac{\sqrt{b^2 -4\left(m+\xi^2\right)}}{2} t} - e^{-\frac{\sqrt{b^2 -4\left(m+\xi^2\right)}}{2} t}}{\sqrt{b^2 -4\left(m+\xi^2\right)}} \, e_{q^2}(-i\xi y) \, e_{q^2}(i\xi x) \, d_q \xi,
$$

\end{thm}

\begin{proof} 
{\it Existence.} Multiplying both sides of (\ref{additive3.1B}) by $\frac{1}{2\pi_q}e_{q^2}(-i\xi x)$ and by $q$-integrating with respect to $x$, we obtain that
\begin{equation}
\label{additive3.3B}
\widehat{u_{tt}}(t, \xi)-\widehat{\mathcal{D}^2_{q,x}u}(t, \xi)+b\widehat{u_t}(t, \xi)+m\widehat{u}(t, \xi)=0, \,\,\, t\in[0, T],
\end{equation}
for all $\xi\in\mathbb R^{+}_{q}$.

Then, by the property (\ref{additive2.3}) we have that
\begin{eqnarray}
\label{additive3.4B}
\widehat{\mathcal{D}^2_{q,x}u}(t, \xi) &=&\frac{1}{2\pi_q}\int\limits_0^\infty \mathcal{D}_{q,x}u(t, x)e_{q^2}(-i\xi x)d_qx
=-\xi^2\widehat{u}(t, \xi),
\end{eqnarray}
for all $\xi\in\mathbb R^{+}_{q}$, for $t\in[0, T]$.

By repeating (\ref{additive3.4}) again, we find that
\begin{eqnarray}
\label{additive3.6B}
\widehat{u}_{tt}(t, \xi)&=&\frac{\partial^2{u}}{\partial{t^2}}\widehat{u}(t, \xi), \,\,\, t\in[0, T], 
\end{eqnarray}
for all $\xi\in\mathbb R^{+}_{q}$.

Now by substituting (\ref{additive3.4}),  (\ref{additive3.4B})  and (\ref{additive3.6B}) into (\ref{additive3.3B}),  we can conclude that
\begin{equation}
\label{additive3.7B}
\left\{
  \begin{array}{ll}
    \widehat{u}_{tt}(t, \xi)+b\widehat{u}_{t}(t, \xi)+(m+\xi^2)\widehat{u}(t, \xi)=0, \,\,\, t\in[0, T], \\
    \widehat{u}(0, \xi)=\widehat{\varphi}(\xi),  \\
    \widehat{u}_{t}(0, \xi)=\widehat{\psi}(\xi),
  \end{array}
\right.
\end{equation}
for all $\xi\in \mathbb{R}_q^+$.

The general solution of (\ref{additive3.7B}) is
\begin{eqnarray}
\label{additive3.8B}
\widehat{u}(t,\xi)=\widehat{G}_1(\xi)e^{(-\frac{b}{2}+\frac{\sqrt{b^2 -4\left(m+\xi^2\right)}}{2})t}+\widehat{G}_2(\xi)e^{(-\frac{b}{2}-\frac{\sqrt{b^2 -4\left(m+\xi^2\right)}}{2})t},
\end{eqnarray}
where $\widehat{G}_1(\xi), \widehat{G}_2(\xi)$ are arbitrary constants.

From the initial conditions  (\ref{additive3.7B}) and (\ref{additive3.8B}) we obtain that 
\begin{equation*}
\begin{split}
\widehat{u}(0, \xi)& =\widehat{G}_1(\xi)+\widehat{G}_2(\xi)=\widehat{\varphi}(\xi), \\
\widehat{u}_t(0,\xi)
& =\left(-\frac{b}{2}+\frac{\sqrt{b^2 -4\left(m+\xi^2\right)}}{2}\right)\widehat{G}_1(\xi)+\left(-\frac{b}{2}-\frac{\sqrt{b^2 -4\left(m+\xi^2\right)}}{2}\right)\widehat{G}_2(\xi) \\
& =\widehat{\psi}(\xi).
\end{split}    
\end{equation*}

Hence,
\begin{eqnarray}\label{additive3.9B}
\widehat{G}_1(\xi)&=&\left[\frac{1}{2}+\frac{b}{2\sqrt{b^2 -4\left(m+\xi^2\right)}}\right]\widehat{\varphi}(\xi)+\frac{1}{\sqrt{b^2 -4\left(m+\xi^2\right)}}\widehat{\psi}(\xi),
\end{eqnarray}
and
\begin{eqnarray}\label{additive3.10B}
\widehat{G}_2(\xi)&=&\left[\frac{1}{2}-\frac{b}{2\sqrt{b^2 -4\left(m+\xi^2\right)}}\right]\widehat{\varphi}(\xi)-\frac{1}{\sqrt{b^2 -4\left(m+\xi^2\right)}}\widehat{\psi}(\xi).
\end{eqnarray}

By using (\ref{additive3.9B}), (\ref{additive3.10B}) and inverse $q^2$-Fourier transform (see Definition 2.1) in (\ref{additive3.8B}), we find that
\begin{equation}
\label{additive3.11B}
\begin{split}
u(t,x)=
\int\limits_0^\infty \Phi(t, x, y) \varphi(y) d_q y +\int\limits_0^\infty \Psi(t, x, y) \psi(y) d_q y, 
\end{split}
\end{equation}
where
\begin{equation*}
\begin{split}
\Phi(t, x, y) = \frac{e^{-\frac{b}{2} t}}{8\pi_q^{2}} \int\limits_0^\infty K(t, \xi) e_{q^2}( - i\xi y) e_{q^2}( i \xi x) d_q\xi,
\end{split}
\end{equation*}
with
$$ 
K(t, \xi) = e^{\frac{\sqrt{b^2 -4\left(m+\xi^2\right)}}{2} t} + e^{ - \frac{\sqrt{b^2 -4\left(m+\xi^2\right)}}{2} t} + b \frac{e^{\frac{\sqrt{b^2 -4\left(m+\xi^2\right)}}{2} t} - e^{ - \frac{\sqrt{b^2 -4\left(m+\xi^2\right)}}{2} t}}{\sqrt{b^2 -4\left(m+\xi^2\right)}},
$$
and 
$$
\Psi(t, x, y) = \frac{e^{-\frac{b}{2} t}}{4\pi_q^{2}}\int\limits_0^\infty \frac{e^{\frac{\sqrt{b^2 -4\left(m+\xi^2\right)}}{2} t} - e^{-\frac{\sqrt{b^2 -4\left(m+\xi^2\right)}}{2} t}}{\sqrt{b^2 -4\left(m+\xi^2\right)}} \, e_{q^2}(-i\xi y) \, e_{q^2}(i\xi x) \, d_q \xi.
$$

Hence , the existence part is proved.

Next we will prove that
\begin{eqnarray}\label{additive3.11B}
u\in C^2\left([0, T]; L^2_q\left(\mathbb{R}^+_q\right) \right)\cap C\left([0, T]; W_q^2\left(\mathbb{R}^+_q\right) \right).
\end{eqnarray}

According to the relations (\ref{additive3.8B}), (\ref{additive3.9B}) and (\ref{additive3.10B}) and taking into account that $\left|\frac{\sqrt{b^2 -4\left(m+\xi^2\right)}}{2}\right|\approx\left(1+\xi^2\right)^\frac{1}{2}\approx\left(1+\xi\right)$, we we can make the following estimates:
\begin{eqnarray*}
\left|\widehat{u}(t,\xi)\right|&\lesssim &\left|\left[\frac{1}{2}-\frac{-\frac{b}{2}}{2\frac{\sqrt{b^2 -4\left(m+\xi^2\right)}}{2}}\right]\widehat{\varphi}(\xi)+
\frac{1}{2\frac{\sqrt{b^2 -4\left(m+\xi^2\right)}}{2}}\widehat{\psi}(\xi)\right|
\\
&\lesssim &\left|\widehat{\varphi}(\xi)\right|+\frac{1}{1+\xi}\left|\widehat{\psi}(\xi)\right|,
\\
\left|\widehat{\partial_tu}(t,\xi)\right|&\lesssim &(1+\xi)\left|\widehat{\varphi}(\xi)\right| +
\left|\widehat{\psi}(\xi)\right|,
\end{eqnarray*}
and
\begin{eqnarray*}
\left|\widehat{\partial^2_tu}(t,\xi)\right|&\lesssim &(1+\xi)^2\left|\widehat{\varphi}(\xi)\right|+
(1+\xi)\left|\widehat{\psi}(\xi)\right|,
\end{eqnarray*}
for all $\xi\in\mathbb R_{q}^{+}$.

Since  $\varphi \in W_q^2\left(\mathbb{R}^+_q\right), \psi\in W_q^1\left(\mathbb{R}^+_q\right)$,
and by using the Parseval's identity and repeating the convergence part of Theorem \ref{TH1}, we arrive at
\begin{eqnarray*}
\|u\|_{C\left([0, T]; W_q^2\left(\mathbb{R}^+_q\right) \right)}
\lesssim\|\varphi\|_{W_q^2\left(\mathbb{R}^+_q\right)}+\|\psi\|_{W^1_q(\mathbb{R}^+_q)}<\infty,
\end{eqnarray*}
and
\begin{eqnarray*}
\|u\|_{C^2\left([0, T]; L^2_q\left(\mathbb{R}^+_q\right) \right)}
\lesssim\|\varphi\|_{W_q^2\left(\mathbb{R}^+_q\right)}+\|\psi\|_{W_q^1\left(\mathbb{R}^+_q\right)}<\infty,
\end{eqnarray*}
and (\ref{additive3.11B}) is proved.

{\it Uniqueness.} It only remains to prove the uniqueness of the solution. We assume the opposite, namely that there exist the functions $u(t,x)$ and $v(t,x)$, which are two different solutions of Problem \ref{Pr2}. Thus, we have that
$$
\left\{
  \begin{array}{ll}
    u_{tt}(t, x)-\mathcal{D}^2_{q,x}u(t, x)+m\nu(t, x)=f(t, x), & (t, x)\in\hbox{$[0, T]\times\mathbb{R}^+_q$,} \\
    u(0, x)=\varphi(x),  u_t(0, x)=\psi(x)  & \hbox{$x\in \mathbb{R}^+_q$,}
  \end{array}
\right.
$$
and
$$
\left\{
  \begin{array}{ll}
    v_{tt}(t, x)-\mathcal{D}^2_{q,x}v(t, x)+mv(t, x)=f(t, x), & (t, x)\in\hbox{$[0, T]\times\mathbb{R}^+_q$,} \\
    v(0, x)=\varphi(x),    v_t(0, x)=\varphi(x),& \hbox{$x\in \mathbb{R}^+_q$.}
  \end{array}
\right.
$$

We define $W(t,x)=u(t,x)-v(t,x)$. Then the function $W(t,x)$ is a solution of the following problem
\begin{eqnarray*}
\left\{
  \begin{array}{ll}
    W_{tt}(t, x)-\mathcal{D}^2_{q,x}W(t, x)+b W_{t}(t, x)+m W(t, x)=0, & (t, x)\in\hbox{$[0, T]\times\mathbb{R}^+_q$,} \\
    W(0, x)=0,   W_t(0, x)=0, & \hbox{$x\in \mathbb{R}^+_q$.}
  \end{array}
\right.
\end{eqnarray*}
From  (\ref{additive3.11B}) it follows that $W(t,x)\equiv0$, that is,  $u(x,t)\equiv v(x,t)$ and this
contradiction to our assumption proves the uniqueness of the solution. The proof is complete.
\end{proof}

\section{Final remark and result}
\label{S4}

\begin{rem} The technique we have developed above can be used to investigate also other $q$-equations. We just present one example of this fact by solving completely the following non-homogeneous equation  (\ref{additive4.1}) with  homogeneous boundary conditions:
\end{rem}

\begin{thm} 
\label{TH3} 
Fix $0<T<\infty$. Let $b>0$ and $m>0$ be such that $b^{2}<4m.$ Assume that $f\in C\left( [0, T]; W^1_q\left(\mathbb{R}^+_q\right)\right)$. Then there exists a unique solution
$$
u\in C^2\left([0, T]; L^2_q\left(\mathbb{R}^+_q\right) \right)\cap C\left([0, T]; W_q^2\left(\mathbb{R}^+_q\right) \right),
$$
of the Cauchy problem
\begin{equation}
\label{additive4.1}  
\begin{split}
& u_{tt}(t, x)-\mathcal{D}^2_{q,x}u(t, x)+bu_{t}(t, x)+mu(t, x)=f(t, x), \,\,\, (t, x)\in\hbox{$[0, T]\times\mathbb{R}^+_q$,} \\
& u(0, x)=0,   u_t(0, x)=0, \,\,\, \,\,\, \,\,\, \,\,\, \,\,\, \,\,\, \hbox{$x\in \mathbb{R}^+_q$.}
\end{split}
\end{equation}

Moreover, the solution can be represented by the formula
\begin{equation}
\label{additive4.2}
u(t,x)=\frac{1}{2\pi_q}
\int\limits_0^t \int\limits_0^\infty K(t, \tau, x, y) f(\tau, y)  d_q y d\tau,
\end{equation}
where 
$$
K(t, \tau, x, y) = \int\limits_0^\infty F(t, \tau, \xi) e_{q^2}( - i \xi y) e_{q^2}(i\xi x) d_q\xi,
$$
with
$$
F(t, \tau, \xi)=\frac{e^{-\frac{b}{2}+\frac{\sqrt{b^2 -4\left(m+\xi^2\right)}}{2}(t-\tau)}-e^{-\frac{b}{2}-\frac{\sqrt{b^2 -4\left(m+\xi^2\right)}}{2}(t-\tau)}}{\sqrt{b^2 -4\left(m+\xi^2\right)}}.
$$
\end{thm}

\begin{proof} {\it Existence.} The first part of the
proof consists more as  less of repeating the  first part of the proof of Theorem \ref{TH2}. In fact, by taking $q^2$-Fourier transform of both sides of (\ref{additive4.1}) and discussing as before, we have that
\begin{eqnarray}\label{additive4.3}
\widehat{u}_{tt}(t)+b\widehat{u}_t(t)+\left(m+\xi^2\right)\widehat{u}(t)=\widehat{f}(t).
\end{eqnarray}

The general solution of the differential equation (\ref{additive4.3}) is
\begin{eqnarray}\label{additive4.4}
\widehat{u}(t,\xi)&=&\widehat{G}_1(\xi)e^{(-\frac{b}{2}+\frac{\sqrt{b^2 -4\left(m+\xi^2\right)}}{2})t}+\widehat{G}_2(\xi)e^{(-\frac{b}{2}-\frac{\sqrt{b^2 -4\left(m+\xi^2\right)}}{2})t}\nonumber\\
&+&\frac{1}{2\frac{\sqrt{b^2 -4\left(m+\xi^2\right)}}{2}}\int\limits_0^t\widehat{f}(\tau, \xi)
e^{(-\frac{b}{2}+\frac{\sqrt{b^2 -4\left(m+\xi^2\right)}}{2})(t-\tau)}d\tau\nonumber\\
&-&\frac{1}{2\frac{\sqrt{b^2 -4\left(m+\xi^2\right)}}{2}}\int\limits_0^t \widehat{f}(\tau, \xi)
e^{(-\frac{b}{2}-\frac{\sqrt{b^2 -4\left(m+\xi^2\right)}}{2})(t-\tau)}d\tau,
\end{eqnarray}
for all $\xi\in\mathbb R_{q}^{+}$.

 Moreover, by differentiating with help of Leibniz's rule we find that (Note that   the factor
 $\frac{1}{2\frac{\sqrt{b^2 -4\left(m+\xi^2\right)}}{2}} \widehat{f}(\tau, \xi)$ cancels):
\begin{eqnarray}
\label{additive4.5}
\widehat{u}_t(t,\xi)&=&(-\frac{b}{2}+\frac{\sqrt{b^2 -4\left(m+\xi^2\right)}}{2})\widehat{G}_1(\xi)e^{(-\frac{b}{2}+\frac{\sqrt{b^2 -4\left(m+\xi^2\right)}}{2})t}\nonumber\\
&+&(-\frac{b}{2}-\frac{\sqrt{b^2 -4\left(m+\xi^2\right)}}{2})
\widehat{G}_2(\xi)e^{(-\frac{b}{2}-\frac{\sqrt{b^2 -4\left(m+\xi^2\right)}}{2})t}\nonumber\\
&+&\frac{-\frac{b}{2}+\frac{\sqrt{b^2 -4\left(m+\xi^2\right)}}{2})}{2\frac{\sqrt{b^2 -4\left(m+\xi^2\right)}}{2}}\int\limits_0^t\widehat{f}(\tau, \xi)
e^{(-\frac{b}{2}+\frac{\sqrt{b^2 -4\left(m+\xi^2\right)}}{2})(t-\tau)}d\tau\nonumber\\
&-&\frac{-\frac{b}{2}-\frac{\sqrt{b^2 -4\left(m+\xi^2\right)}}{2}}{2\frac{\sqrt{b^2 -4\left(m+\xi^2\right)}}{2}}\int\limits_0^t \widehat{f}(\tau, \xi)
e^{(-\frac{b}{2}+\frac{\sqrt{b^2 -4\left(m+\xi^2\right)}}{2})(t-\tau)}d\tau.
\end{eqnarray}

By new using the initial conditions in (\ref{additive4.1}), from (\ref{additive4.4})-(\ref{additive4.5}) it follows that
\begin{eqnarray*}
\widehat{u}(0,\xi)&=&\widehat{G}_1(\xi)+\widehat{G}_2(\xi)=0,\\
\widehat{u}_t(0,\xi)&=&(-\frac{b}{2}+\frac{\sqrt{b^2 -4\left(m+\xi^2\right)}}{2})\widehat{G}_1(\xi)+(-\frac{b}{2}-\frac{\sqrt{b^2 -4\left(m+\xi^2\right)}}{2})\widehat{G}_2(\xi)=0.
\end{eqnarray*}

Therefore, $\widehat{G}_1(\xi)=\widehat{G}_2(\xi)=0$. Finally, we obtain that
\begin{equation}
\label{additive4.5A}
u(t,x)=\frac{1}{2\pi_q}
\int\limits_0^t \int\limits_0^\infty K(t, \tau, x, y) f(\tau, y)  d_q y d\tau,
\end{equation}
where 
$$
K(t, \tau, x, y) = \int\limits_0^\infty F(t, \tau, \xi) e_{q^2}( - i \xi y) e_{q^2}(i\xi x) d_q\xi,
$$
$$
F(t, \tau, \xi)=\frac{e^{(-\frac{b}{2}+\frac{\sqrt{b^2 -4\left(m+\xi^2\right)}}{2})(t-\tau)}-e^{(-\frac{b}{2}-\frac{\sqrt{b^2 -4\left(m+\xi^2\right)}}{2})(t-\tau)}}{\sqrt{b^2 -4\left(m+\xi^2\right)}},
$$
and the existence part is proved.

Next we will show that \begin{equation}\label{additive4.5B}
u\in C\left([0, T]; W_q^2\left(\mathbb{R}^+_q\right)\right)\cap C^2\left([0, T]; L_q^2\left(\mathbb{R}^+_q\right)\right). 
\end{equation}

From (\ref{additive4.4}) and (\ref{additive4.5}), by using the Cauchy-Schwartz inequality, it follows that
\begin{eqnarray*}
\left|\widehat{u}(t,\xi)\right|^2&\lesssim&\left(\int\limits_0^t \left|\frac{\widehat{f}(\tau, \xi)}{1+\xi}\right|
e^{-\frac{b}{2}(t-\tau)}d\tau \right)^{2}
\nonumber\\
&\lesssim& \int\limits_0^t \left|\frac{\widehat{f}(\tau, \xi)}{1+\xi}\right|^{2} d \tau,
\end{eqnarray*}
\begin{eqnarray*}
\left|\widehat{u}_t(t,\xi)\right|^2&\lesssim&\left(\int\limits_0^t \left|\widehat{f}(\tau, \xi)\right|
e^{-\frac{b}{2}(t-\tau)}d\tau \right)^{2}
\nonumber\\
&\lesssim& \int\limits_0^t \left|\widehat{f}(\tau, \xi)\right|^{2}
d\tau,
\end{eqnarray*}
and
\begin{eqnarray*}
\left|\widehat{u}_{tt}(t,\xi)\right|^2&\lesssim&\left(\int\limits_0^t \left|(1+\xi)\widehat{f}(\tau, \xi)\right|
e^{-\frac{b}{2}(t-\tau)}d\tau \right)^{2}
\nonumber\\
&\lesssim& \int\limits_0^t \left|(1+\xi)\widehat{f}(\tau, \xi)\right|^{2} d\tau.
\end{eqnarray*}

Since $f\in C\left([0, T]; W^1_q\left(\mathbb{R}^+_q\right) \right)$,
by new using the Parseval identity and repeating the convergence part of Theorem \ref{TH1}, we arrive at
\begin{eqnarray*}
\|u\|_{C\left([0, T]; W^2_q\left(\mathbb{R}^+_q\right) \right)}
&\lesssim&\|f\|_{C\left([0, T]; W^1_q\left(\mathbb{R}^+_q\right) \right)}<\infty,
\end{eqnarray*}
and
\begin{eqnarray*}
\|u\|_{C^2\left([0, T]; L^2_q\left(\mathbb{R}^+_q\right) \right)}
&\lesssim&\|f\|_{C\left([0, T]; W^1_q\left(\mathbb{R}^+_q\right) \right)}<\infty,
\end{eqnarray*}
so (\ref{additive4.5B}) is proved.

{\it Uniqueness.} To prove the uniqueness we assume by contradiction that both of the different   functions $u(t,x)$ and $v(t,x)$ are   solutions of Problem \ref{Pr2}, that is,
$$
\left\{
  \begin{array}{ll}
    u_{tt}(t, x)-\mathcal{D}^2_{q,x}u(t, x)+b u_t(t, x)+m u(t, x)=f(t, x), & (t, x)\in\hbox{$[0, T]\times\mathbb{R}^+_q$,} \\
    u(0, x)=u_t(0, x)=0,  & \hbox{$x\in \mathbb{R}^+_q$.}
  \end{array}
\right.
$$
and
$$
\left\{
  \begin{array}{ll}
    v_{tt}(t, x)-\mathcal{D}^2_{q,x}v(t, x)+bv_t(t, x)+mv(t, x)=f(t, x), & (t, x)\in\hbox{$[0, T]\times\mathbb{R}^+_q$,} \\
    v(0, x)=v_t(0, x)=0,    & \hbox{$x\in \mathbb{R}^+_q$.}
  \end{array}
\right.
$$

We define   $W(x,t)=u(x,t)-v(x,t)$. Then the function $W(t,x)$ is the solution of the problem
\begin{eqnarray*}
\left\{
  \begin{array}{ll}
    W_{tt}-\mathcal{D}^2_{q,x}W+bW_{t}+mW=0, & (t, x)\in\hbox{$\mathbb{R^+}\times\mathbb{R}^+_q$,} \\
    W|_{t=0}=W_t|_{t=0}=0, & \hbox{$x\in \mathbb{R}^+_q$.}
  \end{array}
\right.
\end{eqnarray*}

From  (\ref{additive4.5A}) it follows that $W(t,x)\equiv0$.  Hence, $u(t,x)\equiv v(t,x)$ which
contradicts to our assumption.  The proof is complete.

\end{proof}

\textbf{Acknowledgment}: The first author was supported by Ministry of Education and Science of the Republic of Kazakhstan Grant AP08052208.

\end{document}